\documentclass[12pt]{article}
\usepackage{amsfonts}
\usepackage{amsmath}
\usepackage{amssymb}
\usepackage{amsthm}
\usepackage{mathrsfs}
\usepackage{graphics}

\newtheorem{theorem}{Theorem}[section]
\newtheorem{lma}[theorem]{Lemma}
\newtheorem{prop}[theorem]{Proposition}

\newtheorem{fact}{Fact}

\theoremstyle{definition}
\newtheorem{defn}[theorem]{Definition}

\newtheorem{ex}[theorem]{Example}

\newcommand{\C}{\mathbb{C}}

\newcommand{\im}{\mathop{{\rm Im}}\nolimits}

\def\adj{\mathop{{\rm adj}}\nolimits}
\def\id{\mathord{{\rm id}}}
\def\imm{\mathop{{\rm imm}}\nolimits}

\def\miximm{\mathop{{\rm miximm}}\nolimits}

\def\tr{\mathop{{\rm tr}}\nolimits}

\title{The $k$-th derivatives of the immanant and the $\chi$-symmetric power of an operator}

 \author{
S\'onia Carvalho\thanks{Centro de Estruturas Lineares e Combinat\'oria da Universidade de Lisboa, Av Prof Gama Pinto 2, P-1649-003 Lisboa and Departamento de Matem\'atica do ISEL, Rua Conselheiro Em\'\i dio Navarro 1, 1959-007 Lisbon, Portugal (scarvalho@adm.isel.pt).}
\and
Pedro J.\ Freitas\thanks{Centro de Estruturas Lineares e Combinat\'oria, Av Prof Gama Pinto 2, P-1649-003 Lisboa and Departamento de Matem\'atica da Faculdade de Ci\^encias, Campo Grande,
Edif\'\i cio C6, piso 2, P-1749-016 Lisboa. Universidade de Lisboa (pedro@ptmat.fc.ul.pt).}
 }
 \date{February, 2013}

\begin{document}

\maketitle

\begin{abstract}
In recent papers, R.\ Bhatia, T.\ Jain and P.\ Grover obtained formulas for directional derivatives, of all orders, of the determinant, the permanent, the $m$-th compound map and the $m$-th induced power map. In this paper we generalize these results for immanants and for other symmetric powers of a matrix. 
\end{abstract} 

\section{Introduction}

There is a formula for the derivative of the determinant map on the space of the square matrices of order $n$, known as the \textit{Jacobi formula}, which has been well known for a long time. In recent work, T.\ Jain and R.\ Bhatia derived formulas for higher order derivatives of the determinant (\cite{BJ}) and T.\ Jain also had derived formulas for all the orders of derivatives for the map $\wedge^m$ that takes an $n \times n$ matrix to its $m$-th compound (\cite{J}).
Later, P.\ Grover, in the same spirit of Jain's work, did the same for the permanent map and the for the map $\vee^m$ that takes an $n \times n$ matrix to its $m$-th induced power. The mentioned authors extended the theory in \cite{BGJ}.

This paper follows along the lines of this work. It is known that the determinant map and the permanent map are special cases of a more generalized map, which is the immanant, and the compound and the induced power of a matrix are also generalized by other symmetric powers, related to symmetric classes of tensors. These will be our objects of study.

\section{Immanant}
We will write $M_n(\mathbb{C})$ to represent the vector space of the square matrices of order $n$ with complex entries. Let $A \in M_n(\mathbb{C})$ and $\chi$ be an irreducible character of  $\mathbb{C}$. We define the immanant of $A$ as: 
$$d_{\chi}(A)=\sum_{\sigma \in S_n} \chi(\sigma)\displaystyle \prod_{i=1}^n a_{i\sigma(i)}.$$
In other words, $d_{\chi}:M_n(\mathbb{C} )\longrightarrow \mathbb{C}$ is a map taking an $n \times n$ matrix to its immanant. This is a differentiable map. For $X\in M_n(\C)$, we denote by $D d_{\chi}A(X)$ the directional derivative of $d_{\chi}$ at $A$ in the direction of $X$.

We denote by $A(i|j)$ the $n\times n$ square matrix that is obtained from $A$ by replacing the $i$-th row and $j$-th column with zero entries, except entry $(i,j)$ which we set to 1. We define the \textit{immanantal adjoint} $\adj_{\chi}(A)$ as the $n\times n$ matrix in which the entry $(i,j)$ is $d_{\chi}(A(i|j))$. This agrees with the definition of permanental adjoint in \cite[p.\ 237]{Me}, but not with the usual adjugate matrix (we would need to consider the transpose in that case). This is a matter of convention.

We obtain the following result similar to the \textit{Jacobi formula} for the determinant.

\begin{theorem}
For each $X \in M_n(\mathbb{C})$, $$D d_{\chi}(A)(X)= \tr (\adj_{\chi}(A)^T X).$$
\end{theorem}

\begin{proof}
For each $1 \leq j \leq n$, let $A(j;X)$ be the matrix obtained from $A$ by replacing the $j$-th column of $A$ by the $j$-th column of $X$ and keeping the rest of the columns unchanged. Then the given equality can be restated as 
\begin{equation}
D d_{\chi}(A)(X) = \sum_{j=1}^n d_{\chi}(A(j;X)).
\label{1}
\end{equation} 

On the other hand we note that $ D d_{\chi}(A)(X)$ is the coefficient of $t$ in the polynomial $d_{\chi}(A+tX)$. Using the fact that the immanant is a multilinear function of the columns we obtain the desired result.
\end{proof}

Again using the fact that the immanant is a multilinear function, we notice that for any $1 \leq i \leq n$, we have 
$$d_{\chi}(A)= \displaystyle \sum_{i=1}^n a_{ij}d_{\chi}(A(i|j)).$$
 
 Using this and (\ref{1}), we get that
 \begin{equation}
 D d_{\chi}(A)(X) = \sum_{i=1}^n \displaystyle \sum_{j=1}^n x_{ij} d_{\chi}(A(i|j)).
 \label{2}
 \end{equation}
 
 We will generalize the expressions (\ref{1}) and (\ref{2}) for the derivatives of all orders of the immanant.

We now turn to derivatives. Let  $V_1,\ldots ,V_n$ be $n$ vector spaces over $\C$, and let $\phi: V_1 \times\ldots \times V_n \longrightarrow \C$ be a multilinear form. For $A, X^1,\ldots X^k \in V_1,\ldots ,V_n$, the $k$-th derivative of $\phi$ at $A$ in the direction of $(X^1,\ldots ,X^k)$ is given by the expression
$$D^k\phi(A)(X^1,\ldots ,X^k):=\frac{\partial^k}{\partial t_1\ldots \partial t_k}\Big \vert_{t_1=\ldots =t_k=0}\phi(A+t_1X^1+\ldots +t_kX^k).$$

This is a multilinear function of $M^k_n(\mathbb{C})$. In particular, if we consider $\phi=d_{\chi}$ we have the definition of the $k$-th derivative of the immanant.

\section{First Expression}

We start by introducing some notation. Given a matrix  $A \in M_n(\mathbb{C})$, we will represent by $A_{[i]}$ the $i$-th column of $A$, $i=\lbrace 1,\ldots ,n \rbrace$. 

Let $k$ be a natural number, $1 \leq k \leq n$. Take $A, X^1,\ldots ,X^k \in M_n(\mathbb{C})$, and $t_1,\ldots ,t_k$ $k$ unknowns. Let $Q_{k,n}$ be the set of strictly increasing maps $\lbrace 1,\ldots ,k \rbrace\longrightarrow\lbrace 1,\ldots ,n \rbrace $ and $G_{k,n}$ the set of increasing maps.

We will denote by $A(\alpha ;X^1,\ldots ,X^k)$ the matrix of order $n$ obtained from $A$ replacing the $\alpha(j)$ column of $A$ by the $\alpha(j)$ column of $X^j$.

\begin{theorem}[First expression]\label{formula1}
For every $1 \leq k \leq n,$ 
$$D^k d_{\chi}(A)(X^1,\ldots ,X^k)= \displaystyle \sum_{\sigma \in S_k} \displaystyle \sum_{\alpha \in Q_{k,n}}d_{\chi}A(\alpha ;X^{\sigma(1)},\ldots ,X^{\sigma(k)}).$$ 
In particular,
$$D^k d_{\chi}(A)(X,\ldots ,X)= k!\displaystyle \sum_{\alpha \in Q_{k,n}}d_{\chi}A(\alpha ;X,\ldots ,X).$$
\end{theorem}

\begin{proof}
Just like in the case of the first derivative, we have that $D^k d_{\chi}(A)(X^1,\ldots ,X^k)$ is the coefficient of $t_1\ldots t_k$ in the expansion of the polynomial $d_{\chi}(A+t_1X^1+\ldots +t_kX^k)$. Using the linearity of the immanant function in each column, we obtain the desired equality.
\end{proof}

\begin{proof}
One simply has to observe that each summand in $\Delta_\chi(A;X^1,\ldots,X^k)$ appears $(n-k)!$ times: once we fix a permutation of the matrices $X^1,\ldots,X^k$, these summands correspond to the possible permutations of the $n-k$ matrices equal to $A$.
\end{proof}


As an immediate consequence of this result, we can obtain another formula for the derivative of order $k$ of the immanant map. This generalizes formula (26) in \cite{BJ}.

\begin{prop}[First expression, rewritten]
\begin{eqnarray}
\label{eq_disc}
D^kd_{\chi} (A)(X^1,\ldots ,X^k) & = & \frac{n!}{(n-k)!} \Delta_{\chi}(A;X^1,\ldots ,X^k).
\end{eqnarray}
\end{prop}

\section{Laplace Expansion for Immanants}
 
We start by generalizing the Laplace expansion, known for the determinant and the permanent, to all immanants. This expansion was proved first for the determinant and the same arguments used can be used to prove the correspondent expansion for the permanent. These can be found in \cite{MaM} and in \cite{Mi}. 

The similarity of proofs is due to the fact that the determinant and the permanent of a direct sum of matrices is just the product of the determinants, or the permanents, of the summand blocks. However, if $\chi$ is any other irreducible character, there is no clear general relation between the immanant of $A$ and the immanant of any submatrix of $A$. So the Laplace expansion formula for any immanant is a little more complicated.\medskip
 

 Let $\alpha \in Q_{k,n}$, and 
 $$S_{k_\alpha}=\lbrace \sigma \in S_n : \sigma (\lbrace 1,\ldots ,k \rbrace)=\im \alpha \rbrace.$$
 \begin{fact} If $\alpha \neq \beta$ then $S_{k_\alpha} \cap S_{k_\beta}=\emptyset.$
 \end{fact}
 Suppose $\sigma \in S_{k_\alpha} \cap S_{k_\beta}$. Then  $\sigma (\lbrace 1,\ldots ,k \rbrace)=\im \alpha=\im \beta$, and thus $\im \alpha=\im \beta$. Since $\alpha, \beta \in Q_{k,n}$, it follows that $\alpha= \beta$.
 
 \begin{fact}$$S_n =\displaystyle \bigcup_{\alpha \in Q_{k,n}}^{\centerdot}S_{k_\alpha}$$
 \end{fact}
 
 If $\pi \in S_n$ then $\pi( \lbrace 1,\ldots ,k \rbrace)= \lbrace j_1,\ldots ,j_k \rbrace.$ Suppose that $j_1 <\ldots <j_k.$
 
 Let $\gamma \in S_n$ such that $\gamma(i)=j_i$. Therefore $\pi \in S_{k_\gamma}$ and $S_n \subseteq \displaystyle \bigcup_{\alpha \in Q_{k,n}} S_{k_\alpha}$.
 
 The other inclusion is trivial.
 
 \begin{fact} $|\lbrace S_{k_\alpha}: \alpha \in Q_{k,n}\rbrace| = | Q_{k,n}|= \dfrac{n!}{k!(n-k)!}$.
 \end{fact}
 
  \begin{fact} If $\sigma \in S_{k_\alpha}$ then $\sigma (\lbrace k+1,\ldots ,n \rbrace)=\overline{\im \alpha}$ (the complement of $\im \alpha$).
 \end{fact}
 
 Suppose that $l \in  \lbrace k+1,\ldots ,n \rbrace$ and $\sigma(l)=j_l \in \im \alpha$. We have $\sigma \in S_{k_\alpha}$ so we can find $i \in \lbrace j_1,\ldots ,j_k \rbrace$ such that $\sigma(i)=j_l$. But $i \neq l$. This is a contradiction, because  $\sigma \in S_n$, and so $\sigma$ is injective.\medskip

 Using the previous facts, we can conclude that for every  $\sigma \in S_{k_\beta}$, the value
 $$\displaystyle \sum_{\sigma \in S_{k_\beta} }\chi(\sigma)\displaystyle\prod_{t=1}^n a_{t\sigma(t)}$$ 
 does not depend on the values of the following entries of the matrix $A$:
 \begin{enumerate}
 \item[I.] The first $k$ rows of  $A$ and the $n-k$ columns of $A$ in $\overline{\im \beta}$.
 \item[II.] Rows $k+1,\ldots ,n$ of $A$ and the $k$ columns of $A$ with index in $\im \beta$.
 \end{enumerate}
 
We now denote by $A\{\id | \beta\}=(a^*_{ij})$ the matrix of order $n$ obtained by replacing every entry in I and II by zeros. We then have
 $$\displaystyle \sum_{\sigma \in S_{k_\beta} }\chi(\sigma)\displaystyle\prod_{t=1}^n a_{t\sigma(t)}=d_\chi (A\{\id|\beta\}).$$
 On the other hand,
 \begin{eqnarray*}
 d_\chi(A\{\id | \beta\})&=&\displaystyle \sum_{\sigma \in S_n }\chi(\sigma)\displaystyle\prod_{t=1}^n a^*_{t\sigma(t)} \\
 &=&\displaystyle \sum_{\sigma \in \bigcup_{\alpha \in Q_{k,n}}S_{k_\alpha}  }\chi(\sigma)\displaystyle\prod_{t=1}^n a^*_{t\sigma(t)}.
\end{eqnarray*}
  
Now take $\alpha\neq \beta$ and $\sigma \in S_{k_{\alpha}}$. Then $\displaystyle\prod_{t=1}^n a^*_{t\sigma(t)}=0$, because at least one of the factors is zero. So 
$$\sum_{\sigma \in S_{k_{\alpha}} }\chi(\sigma)\displaystyle\prod_{t=1}^n a^*_{t\sigma(t)}=0,$$ 
for every  $\alpha\in Q_{k,n}\setminus\{\beta\}$.
 
Moreover, for  $A\{\id | \beta\}$, if $\sigma \in S_{k_\beta}$ then $a^*_{t\sigma(t)}=a_{t\sigma(t)}$.  So $d_\chi(A\{\id |\beta\})=\sum_{\sigma \in S_{k_{\beta}}}\chi(\sigma)\displaystyle\prod_{t=1}^n a_{t\sigma(t)}$. Finally,
\begin{eqnarray}
d_\chi(A)&=& \sum_{\sigma \in S_n}\chi(\sigma)\prod_{t=1}^n a_{t\sigma(t)}\nonumber \\
&=&\displaystyle \sum_{\alpha\in Q_{k,n}} \displaystyle \sum_{\sigma \in S_{k_\alpha}}\chi(\sigma)\displaystyle\prod_{t=1}^n a_{t\sigma(t)}\nonumber \\
&=& \displaystyle \sum_{\alpha\in Q_{k,n}} d_\chi(A\{\id |\alpha\}).\label{1formula}
\end{eqnarray}

\begin{defn}
Let $X$ be a square matrix of order  $n$, $k$ a natural number, $1 \leq k \leq n$ e $\alpha, \beta \in Q_{k,n}$. We denote $X[\alpha|\beta]$ as the square matrix of order $k$ obtained from $X$ by picking the rows $\alpha(1),\ldots ,\alpha(k)$ and the columns $\beta(1),\ldots ,\beta(k)$.

We denote $X(\alpha|\beta)$ as the square matrix of order $n-k$ obtained from $X$ by deleting the rows $\alpha(1),\ldots ,\alpha(k)$ and the columns $\beta(1),\ldots ,\beta(k)$.
\end{defn}

Let $\alpha, \beta \in Q_{k,n}$, and take $A$ a square matrix of order $k$ and $B$ a square matrix of order  $n-k$. Denote by $\bar \alpha$ be the unique element of $Q_{n-k,n}$ with $\im \overline \alpha = \overline{\im \alpha}$. 

We now define
$$A\bigoplus_{\alpha | \beta}B=(x_{ij}),$$
as a square matrix of order $n$ such that
\begin{itemize}
\item $x_{ij}=0$ if $i \in \im \alpha$ and $j \not\in \im \beta$;
\item $x_{ij}=0$ if   $i \not \in \im \alpha$  and $j \in \im \beta$;
\item $x_{ij}=a_{\alpha^{-1}(i)\beta^{-1}(j)}$ if $i \in \im \alpha$ and $j \in \im \beta$;
\item $x_{ij}=b_{\overline{\alpha}^{-1}(i)\overline{\beta}^{-1}(j)}$ if $i \not \in \im \alpha$ and $j \not \in \im \beta$.
\end{itemize}

In a sense, we pick rows $\alpha$ and $\beta$ of $A$ and rows $\bar \alpha$ and $\bar \beta$ of $B$. If $\alpha=\beta=(1,\ldots,k)$, this is the usual direct sum of $A[1,\ldots,k|1,\ldots,k]$ and $B(1,\ldots,k|1,\ldots,k)$.\medskip

So, according to our definitions, it is easy to verify that 
$$X\{\id |\beta\}=X[\id|\beta]\bigoplus_{\id | \beta}X(\id|\beta).$$

Applying the previous equality  and  \eqref{1formula} it follows that
$$d_\chi(X)=\displaystyle \sum_{\beta \in Q_{k,n}} d_\chi(X[\id|\beta]\bigoplus_{\id | \beta}X(\id|\beta))
= \displaystyle \sum_{\beta \in Q_{k,n}}d_\chi(X\{\id|\beta\}).$$

For a fixed $\alpha \in Q_{k,n}$, using similar arguments, we get {\em the Laplace expansion for immanants}: 
\begin{equation}
d_\chi(X) = \displaystyle \sum_{\beta \in Q_{k,n}}d_\chi(X[\alpha|\beta]\bigoplus_{\alpha | \beta}X(\alpha|\beta)) = \displaystyle \sum_{\beta \in Q_{k,n}}d_\chi(X\{\alpha|\beta\}).\label{Laplace_immanants}
\end{equation}

\begin{ex}
Let $A$ be a matrix of order $4$. Let $k=2$ and $\alpha=\id$. Then, we have
\begin{eqnarray*}
 d_{\chi}(A)& = & d_{\chi}\begin{pmatrix}
a_{11}& a_{12} & 0 & 0 \\
a_{12}& a_{22} & 0 & 0 \\
0 & 0 & a_{33}& a_{34} \\
0 & 0 & a_{43}& a_{44}
\end{pmatrix} + d_{\chi}\begin{pmatrix}
a_{11}& 0 & a_{13} & 0 \\
a_{12}& 0 & a_{23} & 0 \\
0 & a_{32}& 0 & a_{34} \\
0 & a_{42}& 0 & a_{44}
\end{pmatrix}\\
& + & d_{\chi}\begin{pmatrix}
a_{11}& 0 & 0 & a_{14}\\
a_{12}& 0 & 0 & a_{24}\\
0 & a_{32} & a_{33}& 0 \\
0 & a_{42}& a_{43} & 0
\end{pmatrix}
+ d_{\chi}\begin{pmatrix}
0 & a_{12} & a_{13} & 0 \\
0 & a_{22} & a_{23} & 0 \\
a_{31} & 0 & 0 & a_{34} \\
a_{41} & 0 & 0 & a_{44}
\end{pmatrix}\\
& + & d_{\chi}\begin{pmatrix}
0 & a_{12} & 0 & a_{14} \\
0 & a_{22} & 0 & a_{24} \\
a_{31} & 0& a_{33}& 0 \\
a_{41}&0& a_{43}& 0
\end{pmatrix} + d_{\chi}\begin{pmatrix}
0 & 0 & a_{13} & a_{14}\\
0 & 0 & a_{23} & a_{24}\\
a_{31} & a_{32} & 0 & 0 \\
a_{41} & a_{42}& 0 & 0
\end{pmatrix}
\end{eqnarray*} 
\end{ex}

We list some properties of the matrix $X\{\alpha|\beta\}$. 

\begin{prop} Let $\alpha, \beta, \alpha', \beta'\in Q_{k,n}$. Then we have 
$$X\{\alpha|\beta\}[\alpha|\beta] = X[\alpha|\beta] \text{ and } X\{\alpha|\beta\}(\alpha|\beta) = X(\alpha|\beta).$$ 
If $\beta\neq \beta'$, then both matrices $X\{\alpha|\beta\}[\alpha|\beta']$ and $X\{\alpha|\beta\}(\alpha|\beta')$ have a zero column.  If $\alpha\neq \alpha'$, then both matrices $X\{\alpha|\beta\}[\alpha'|\beta]$ and $X\{\alpha|\beta\}(\alpha'|\beta)$ have a zero row. 
\label{alfa-beta-soma}
\end{prop}
\begin{proof}
These are consequences of the definitions. 
\end{proof}

We can now check that this formula generalises the known Laplace formulas for the determinant and the permanent (see \cite{MaM} and \cite{Mi}).
If $\chi= \varepsilon$ then $d_\varepsilon=\det$. For $\alpha\in Q_{k,n}$, denote $|\alpha|=\alpha(1)+\ldots+\alpha(k)$.  Fixing $\alpha\in Q_{k,n}$ and applying the previous properties we have:
\begin{eqnarray}
\det X & = & \sum_{\beta \in Q_{k,n}}\det(X\{\alpha |\beta\})\nonumber \\
&=& (-1)^{|\alpha|}\sum_{\beta \in Q_{k,n}}\sum_{\gamma \in Q_{k,n}}(-1)^{|\gamma|}\det(X\{\alpha |\beta\}[\alpha|\gamma])\det(X\{\alpha |\beta\}(\alpha|\gamma))\nonumber \\
&=& (-1)^{|\alpha|}\sum_{\beta \in Q_{k,n}}(-1)^{|\beta|}\det(X\{\alpha |\beta\}[\alpha|\beta])\det(X\{\alpha |\beta\}(\alpha|\beta))\nonumber \\
&=& (-1)^{|\alpha|}\sum_{\beta \in Q_{k,n}}(-1)^{|\beta|}\det(X[\alpha|\beta])\det(X(\alpha|\beta)).\label{gen-det-laplace}
\end{eqnarray}

This is exactly the expression of the the Laplace expansion for determinants. With similar arguments we can prove the same result for the Laplace expansion of the permanent.

\section{Second Expression}

We now wish to separate the entries of the matrices $X^1,\ldots,X^k$ from the entries of $A$. It is easy to express the determinant of a direct sum in terms of determinants of direct summands, and the same happens with the permanent. With immanants, the best one can do is use formula \eqref{Laplace_immanants}, which is what we do in this second expression. 

Take $X^1,\ldots ,X^k$ complex matrices of order $n$, and, for $\sigma\in S_k$, $ \beta \in Q_{k,n}$. Denoting by $\mathbf{0}$ the zero matrix of order $n$, we define
$$X^\sigma_\beta = \mathbf{0}(\beta; X^{\sigma(1)},\ldots, X^{\sigma(k)}),$$
the matrix whose $\beta(p)$-th column is equal to $X^{\sigma(p)}_{[\beta (p) ]}$ and the remaining columns are zero, for $1 \leq p \leq k$.

\begin{theorem}[Second Formula]
\label{formula2}
$$D^k d_\chi A(X^1,\ldots ,X^k)= \sum_{\sigma \in S_k}\sum _{\alpha, \beta \in Q_{k,n}}d_\chi (X^\sigma_{\beta}[ \alpha | \beta ] \bigoplus_{\alpha | \beta}A(\alpha|\beta)),$$ 
in particular
$$D^k d_\chi A(X,\ldots ,X)= k!\sum _{\alpha, \beta \in Q_{k,n}}d_\chi (X [ \alpha | \beta ] \bigoplus_{\alpha |\beta}A(\alpha|\beta))$$
\end{theorem}

\begin{proof}
We have proved that 
$$D^k d_\chi A(X^1,\ldots ,X^k)=\sum_{\sigma \in S_k}\sum _{ \beta \in Q_{k,n}}d_\chi A(\beta; X^{\sigma(1)},\ldots ,X^{\sigma(k)}).$$
By the Laplace expansion for immanants, for every  $\beta \in Q_{k,n}$, we have that
\begin{eqnarray*}
&& d_\chi A(\beta; X^{\sigma(1)},\ldots ,X^{\sigma(k)}) =\\
&& = \sum _{\alpha \in Q_{k,n}}d_\chi( A ( \beta; X^{\sigma(1)},\ldots ,X^{\sigma(k)})\{ \alpha|\beta\})\\
&& = \sum _{\alpha \in Q_{k,n}}d_\chi( A ( \beta; X^{\sigma(1)},\ldots ,X^{\sigma(k)})[ \alpha|\beta]\bigoplus_{\alpha | \beta} A( \beta; X^{\sigma(1)},\ldots ,X^{\sigma(k)})( \alpha|\beta) ).
\end{eqnarray*}
Now we just notice that
$$A ( \beta; X^{\sigma(1)},\ldots ,X^{\sigma(k)})[ \alpha|\beta]=X^\sigma_\beta [ \alpha | \beta ]$$ and 
$$A ( \beta; X^{\sigma(1)},\ldots ,X^{\sigma(k)})( \alpha|\beta)=A(\alpha| \beta).$$
This concludes the proof of the formula.\end{proof}

\section{Formulas for the $k$-th derivative for the $m$-th $\chi$-symmetric tensor power}

In this section we wish to establish a formula for the $k$-th derivative of the $\chi$-symmetric tensor power of a matrix. Before we can do this, we need quite a bit of definitions, including the very definition of this matrix. \medskip

We start with some classical results that can be found in \cite[chapter 6]{Me}. Let $\chi$ be an irreducible character of $S_m$ and $$K_{\chi}=\dfrac{\chi(\id)}{m!}\displaystyle \sum_{\sigma \in S_m} \chi(\sigma)P(\sigma),$$ where $\id$ stands for the identity element of $S_m$. $K_{\chi}$ is a linear operator on $\otimes^m V$, and it is also an orthoprojector. Is is called a {\em symmetriser map}.
The range of $K_{\chi}$ is called the {\em symmetry class of tensors}\/ associated with the irreducible character $\chi$ and it is represented by $V_{\chi}=K_\chi(\otimes^m V)$.

It is well known that the alternating character $\chi(\sigma)=\varepsilon_{\sigma}$ (sign of the permutation $\sigma$) leads to the symmetry class $\wedge^m V$; on the other hand the principal character $\chi(\sigma)\equiv 1$ leads to the symmetry class $\vee^m V$.

Given a symmetriser map $K_{\chi}$, we denote 
$$v_1\ast v_2 \ast\ldots  \ast v_m=K_{\chi}(v_1 \otimes v_2 \otimes\ldots  \otimes v_m).$$
These vectors belong to $V_{\chi}$ and are called {\em decomposable symmetrised tensors}.

Let $\Gamma_{m,n}$ be the set of all maps from the set $ \lbrace 1,\ldots ,m\rbrace$ into the set $ \lbrace 1,\ldots ,n\rbrace$. This set can also be identified with the collection of multiindices $\lbrace (i_1,\ldots ,i_m): i_j \leq n \rbrace$. If $\alpha \in\Gamma_{m,n}$, this correspondence associates to $\alpha$ the $m$-tuple $(\alpha(1),\ldots ,\alpha(m))$. In the set $\Gamma_{m,n}$ we will consider the lexicographic order.
The set 
$$ \lbrace \alpha \sigma : \sigma \in S_m \rbrace \subseteq \Gamma_{m,n}$$
is the orbit of $\alpha.$
The group $S_m$ acts on $\Gamma_{m,n}$ by the action $(\sigma, \alpha) \longrightarrow \alpha \sigma^{-1}$ where $\sigma \in S_m$ and $\alpha \in \Gamma_{m,n}$. The stabiliser of $\alpha$ is the subgroup of $S_m$ defined as $$ G_{\alpha}=\lbrace \sigma \in S_m : \alpha \sigma= \alpha \rbrace.$$
Let $\lbrace e_1,\ldots ,e_n \rbrace$ be an orthonormal basis of the vector space $V$. Then $$\lbrace e_{\alpha}^\otimes=e_{\alpha(1)}\otimes e_{\alpha(2)}\otimes\ldots  \otimes e_{\alpha(m)}: \alpha \in \Gamma_{m,n} \rbrace$$ is a basis of the $m$-th tensor power of $V$. So the set of all decomposable symmetrised tensors spans $V_{\chi}$. However, this set need not to be a basis of $V_{\chi}$, because its elements might not be linearly independent, some of them may even be zero. Let 
\begin{equation}
\label{omega_chi}
\Omega =\Omega_{\chi}= \lbrace \alpha \in \Gamma_{m,n} : \sum_{\sigma \in G_{\alpha}}\chi(\sigma) \neq 0 \rbrace.
\end{equation}
With simple calculations, we can conclude that 
\begin{equation} 
\Vert e_{\alpha}^{\ast}\Vert^2 = \dfrac{\chi(\id)}{m!}\sum_{\sigma \in G_{\alpha}}\chi(\sigma).
\label{norma-tensores}
\end{equation}
So the nonzero decomposable symmetrised tensors are $\lbrace e_{\alpha}^{\ast}: \alpha \in \Omega \rbrace.$
Now, let $\Delta$ be the system of distinct representatives for the quocient set
$\Gamma_{m,n}/{S_m}$, constructed by choosing the first element in each orbit, for the lexicographic order of indices. It is easy to check that $\Delta \subseteq G_{m,n}$, where $G_{m,n}$ is the set of all increasing sequences of $\Gamma_{m,n}.$ Let $$\overline{\Delta}=\Delta\cap \Omega.$$ It can be proved that the set $\lbrace e_{\alpha}^{\ast}: \alpha \in \overline{\Delta} \rbrace$ is linearly independent. We have already seen that the set $\lbrace e_{\alpha}^{\ast}: \alpha \in \Omega \rbrace,$ spans $V_{\chi}$, so there is a set $\widehat{\Delta}$, such that $\overline{\Delta}\subseteq \widehat{\Delta} \subseteq \Omega$ and 
\begin{equation}
{\cal E}':=\lbrace e_{\alpha}^{\ast}: \alpha \in \widehat{\Delta} \rbrace,
\end{equation}
is a basis for $V_{\chi}$. It is also known that this basis is orthogonal if $\chi$ is a linear character.\medskip

In general, if $\chi$ does not have degree one, there are no known orthonormal bases of $V_\chi(S_m)$ formed by decomposable symmetrised tensors.  Let $ \mathcal{E}=(v_\alpha: \alpha \in \widehat \Delta)$ be the orthonormal basis of the  $m$-th $\chi$-symmetric tensor power of the vector space $V$ obtained by applying the Gram-Schmidt orthonormalization procedure to $\cal E'$. Let $B$ be the $t \times t$ change of basis matrix, from $\mathcal{E}$ to $\mathcal{E'}=(e^*_{\alpha}:\alpha \in \widehat{\Delta})$. This means that for each $\alpha\in \widehat \Delta$, 
$$v_\alpha = \sum_{\gamma \in \widehat \Delta} b_{\gamma\alpha} e^*_\gamma.$$

We note that this matrix $B$ {\em does not depend on the choice of the orthonormal basis of $V$}, since the set $\widehat \Delta$ is independent of the vectors, and has a natural order (the lexicographic order), which the basis ${\cal E}$ inherits. Moreover, the Gram-Schmidt process only depends on the numbers $\langle e_\alpha^*,e_\beta^*\rangle$, and, by \cite[p.\ 163]{Me}, these are given by formula 
$$\langle e_\alpha^*, e_\beta^* \rangle = \frac{\chi(\id)}{m!} \sum_{\sigma \in S_m} \chi(\sigma) \prod_{t=1}^m \langle e_{\alpha(t)}, e_{\beta\sigma(t)} \rangle.$$
Hence, they only depend on the values of $\langle e_i,e_j\rangle = \delta_{ij}$ and are thus independent of the vectors themselves. \medskip

It is known that if $T$ is a linear operator on V, then $V_\chi$ is invariant for its $m$-th fold tensor power $\otimes^mT$. Thus, the $\chi$-symmetric power $T$, denoted by $K_\chi(T)$ is defined as the restriction of $\otimes^m T$ to $V_\chi$. There is a close connection of this $\chi$-symmetric tensor power of $T$ and the immanant, as the following result already shows (this is a rephrasing of \cite[p.\ 230]{Me}).

\begin{theorem}
Suppose $\chi$ is an irreducible character of the group $S_m$. Let $E=\lbrace e_1,\ldots ,e_n \rbrace$ be an orthonormal basis of the inner product space $V$. Let $T \in L(V,V)$ be the unique linear operator such that $M(T,E)=A.$

If $\alpha, \beta \in \Gamma_{m,n}$, then 
\begin{equation}
\langle K_{\chi}(T) 
(e^*_{\alpha}),e^*_{\beta} \rangle = \frac{\chi(\id)}{m!} d_{\chi}(A^T[ \alpha | \beta ]).
\label{entries_chipower}
\end{equation}
\end{theorem} 

As we came across the immanant of a transpose, we note that $d_\chi(A^T) = d_{\overline{\chi}} (A)$, where $\overline\chi$ is also an irreducible character of $S_m$, defined as $\overline{\chi}(\sigma):=\overline{\chi(\sigma)}$. 
\begin{eqnarray*}
d_\chi(A^T) & = & \sum_{\sigma\in S_m} \chi(\sigma) \prod_{i=1}^m (A^T)_{i\sigma(i)}\\
& = & \sum_{\sigma\in S_m} \chi(\sigma) \prod_{i=1}^m a_{\sigma(i) i}\qquad (i=\sigma^{-1}(j))\\
& = & \sum_{\sigma\in S_m} \chi(\sigma) \prod_{j=1}^m a_{j \sigma^{-1}(j)} \qquad (\sigma=\tau^{-1})\\
& = & \sum_{\tau\in S_m} \chi(\tau^{-1}) \prod_{j=1}^m a_{j \tau(j)}\\
& = & \sum_{\tau\in S_m} \overline{\chi(\tau)} \prod_{j=1}^m a_{j \tau(j)}\\
& = & d_{\overline\chi}(A)
\end{eqnarray*}

*** $\overline{\chi}$ tem o mesmo $\widehat\Delta$ que o $\chi$? [E' so' curiosidade]\\

Now we want to define $K_{\chi}(A)$, the $m$-th $\chi$-symmetric tensor power of the matrix $A$. A natural way to do this is to fix an orthonormal basis $E$ in $V$, and consider the linear endomorphism $T$ such that $A=M(T,E)$. Then the basis $\cal E$ is orthonormal and one can define 
$$K_\chi(A):=M(K_\chi(T),{\cal E})$$
The matrix $K_\chi(A)$ has order $t=|\widehat \Delta|$, with  $|Q_{m,n}| \leq t$.\medskip

It is important to notice that {\em this matrix does not depend on the choice of the orthonormal basis $E$ of $V$}. This is an immediate consequence of the formula \eqref{entries_chipower}: for $\alpha, \beta \in \widehat{\Delta}$, the $(\alpha, \beta)$ entry of $K_{\chi}(A)$ is
\begin{eqnarray*}
\left\langle K_{\chi}(T)v_{\beta},v_{\alpha}\right\rangle & = &
\displaystyle \sum_{\gamma ,  \delta \in \widehat{\Delta}}\left\langle b_{\gamma\beta}K_{\chi}(T)e^*_{\gamma},b_{\delta\alpha}e^*_{\delta}\right\rangle \\
& = & \displaystyle \sum_{\gamma ,  \delta \in \widehat{\Delta}}b_{\gamma\beta} \overline{b_{\delta\alpha}}\left\langle K_{\chi}(T)e^*_{\gamma},e^*_{\delta}\right\rangle \\
& = & \frac{\chi(\id)}{m!}\displaystyle \sum_{\gamma,\delta \in \widehat{\Delta}}b_{\gamma\beta} \overline{b_{\delta\alpha}}d_{\chi}(A^T[\gamma | \delta])\\
& = & \frac{\chi(\id)}{m!}\displaystyle \sum_{\gamma,\delta \in \widehat{\Delta}}b_{\gamma\beta} \overline{b_{\delta\alpha}}d_{\chi}(A[ \delta|\gamma]^T)\\
& = & \frac{\chi(\id)}{m!}\displaystyle \sum_{\gamma,\delta \in \widehat{\Delta}}b_{\gamma\beta} \overline{b_{\delta\alpha}}d_{\overline{\chi}}(A[\delta | \gamma ]).
\end{eqnarray*}

This definition admits, as special cases, the $m$-th compound and the $m$-th induced power of a matrix, as defined in \cite[p.\ 236]{Me}. The matrix $K\chi(A)$ is called the {\em induced matrix\ } in \cite[p.\ 235]{Me}, in the case when the character has degree one.\medskip 

Denote by $\imm_{\chi} (A)$ the square matrix with rows and columns indexed by $\widehat \Delta$, whose $(\gamma, \delta)$ entry is $d_{\chi}(A[ \gamma | \delta])$ (one could call the elements of this matrix {\em immanantal minors\ }�indexed by $\widehat \Delta$, the usual minors are gotten by considering the alternating character, in which case $\widehat \Delta= Q_{m,n}$). With this definition, we can rewrite the previous equation as 
\begin{equation}
K_{\chi}(A)= \frac{\chi(\id)}{m!} B^* \imm_{\overline\chi} (A) B.
\label{deriv-adj}
\end{equation}

Finally, denote by $\miximm_{\chi}(X^1,\ldots,X^n)$ the square matrix with rows and columns indexed by $\widehat \Delta$, whose $(\gamma,\delta)$ entry is $\Delta_{\chi}( X^1[\gamma | \delta],\ldots ,X^n [ \gamma | \delta])$, so that $\miximm_{\chi} (A,\ldots,A)=\imm_{\chi}(A)$. We use the same shorthand as with the mixed immanant: for $k\leq n$, 
$$\miximm_{\chi}(A;X^1,\ldots,X^k):=\miximm_\chi (A,\ldots,A,X^1,\ldots X^k)$$

Before our main formula, we recall a general result about derivatives, which follows from the definition. 

\begin{lma}
If $f$ and $g$ are two maps such that $f\circ g$ is well defined, with $g$ linear, then
$$D^k(f\circ g)(A)(X^1,\ldots,X^{k}) = D^k f(g(A)) (g(X^1),\ldots,g(X^k)).$$
\end{lma}

\begin{theorem} \label{der_pot} According to our previous notation, we have
$$D^k K_{\chi}(A)(X^1,\ldots ,X^k)=\frac{\chi(\id)}{(m-k)!} B^* \miximm_{\overline{\chi}}(A;X^1,\ldots,X^k) B$$
and, using the notation we have already established, the $(\alpha,\beta)$ entry of this matrix is
$$\frac{\chi(\id)}{m!}\displaystyle \sum_{\gamma ,  \delta \in \widehat{\Delta}}b_{\gamma\beta} \overline{b_{\delta\alpha}}\displaystyle \sum_{\sigma \in S_k}\displaystyle \sum_{\rho, \tau \in Q_{k,m}}d_{\overline\chi}(X[\delta | \gamma]^{\sigma}_{\tau}  [ \rho| \tau]\bigoplus_{\rho| \tau} A[\delta|\gamma](\rho | \tau)).$$
\end{theorem}
\begin{proof} Notice that the map $A\mapsto A[\delta|\gamma]$ is linear, so we can apply the previous lemma in order to compute the derivatives of the entries of the matrix $K_{\chi}(A)$. The $(\alpha,\beta)$ entry of the $k$-th derivative of the $m$-th $\chi$-symmetric tensor power of $A$, i.e., the $(\alpha,\beta)$ entry of the matrix $D^k K_{\chi}(A)(X^1,\ldots ,X^k)$ is:
$$
 \frac{\chi(\id)}{k!}\sum_{\gamma ,\delta \in \widehat{\Delta}}b_{\gamma\beta} \overline{b_{\delta\alpha}}D^kd_{\overline\chi}(A[\delta | \gamma] )(X^1[\delta | \gamma] ,\ldots ,X^k[\delta | \gamma]).$$

To abbreviate notation, for fixed $\gamma, \delta\in \widehat \Delta$, we will write $C:=A[ \delta | \gamma]$, and $Z^i:=X^i[\delta | \gamma]$, $i=1,\ldots,k$. Using formula \eqref{eq_disc}, we get
\begin{eqnarray*} 
D^kd_{\overline\chi}(A[\delta | \gamma] )(X^1[\delta | \gamma] ,\ldots ,X^k[\delta | \gamma]) & = & D^kd_{\overline\chi} (C)(Z^1,\ldots ,Z^k)\\
& = & \frac{m!}{(m-k)!} \Delta_{\overline\chi}(C;Z^1,\ldots ,Z^k).
\end{eqnarray*}
So the $(\alpha,\beta)$ entry of $D^k K_{\chi}(A)(X^1,\ldots ,X^k)$ is
\begin{eqnarray*}
\frac{\chi(\id)}{m!}\displaystyle \sum_{\gamma ,  \delta \in \widehat{\Delta}}b_{\gamma\beta} \overline{b_{\delta\alpha}}\frac{m!}{(m-k)!} \Delta_{\overline\chi}(C;Z^1,\ldots ,Z^k)=\\
\dfrac{\chi(\id)}{(m-k)!}\displaystyle \sum_{\gamma ,  \delta \in \widehat{\Delta}}b_{\gamma\beta} \overline{b_{\delta\alpha}} \Delta_{\overline\chi}(A[\delta | \gamma]; X^1[ \delta|\gamma],\ldots ,X^k [\delta|\gamma])
\end{eqnarray*}

According to the definition of $\miximm_{\overline{\chi}}(A;X^1,\ldots,X^k)$, we have
$$D^k K_{\chi}(A)(X^1,\ldots ,X^k)=\frac{\chi(\id)}{(m-k)!} B^* \miximm_{\overline{\chi}}(A;X^1,\ldots,X^k) B.$$

This establishes the first formula. For the entries of the matrix, we use the formula in theorem \ref{formula2}:
$$ D^k d_{\overline\chi}(C)(Z^1,\ldots ,Z^k) = \displaystyle \sum_{\sigma \in S_k}\displaystyle \sum_{\rho,\tau \in Q_{k,m}}d_{\overline\chi}(Z^{\sigma}_{\tau} [ \rho|\tau]\bigoplus_{\rho | \tau} C(\rho|\tau)),$$
recalling that
$$Z^\sigma_\tau = {\mathbf 0} (\tau; Z^{\sigma(1)},\ldots,Z^{\sigma(k)}),$$
where $\mathbf 0$ denotes the zero matrix of order $m$.

So, the $(\alpha,\beta)$ entry of the $k$-th derivative of $K_{\chi}(A)$ is:
\begin{eqnarray*}
\frac{\chi(\id)}{m!}\displaystyle \sum_{\gamma ,  \delta \in \widehat{\Delta}}b_{\gamma\beta} \overline{b_{\delta\alpha}}\displaystyle \sum_{\sigma \in S_k}\displaystyle \sum_{\rho, \tau \in Q_{k,m}}d_{\overline\chi}(Z^{\sigma}_{\tau} [ \rho| \tau]\bigoplus_{\rho| \tau} C(\rho | \tau))= \\ 
\frac{\chi(\id)}{m!}\displaystyle \sum_{\gamma ,  \delta \in \widehat{\Delta}}b_{\gamma\beta} \overline{b_{\delta\alpha}}\displaystyle \sum_{\sigma \in S_k}\displaystyle \sum_{\rho, \tau \in Q_{k,m}}d_{\overline\chi}(X[\delta | \gamma]^{\sigma}_{\tau} [ \rho| \tau]\bigoplus_{\rho| \tau} A[\delta|\gamma](\rho | \tau))
\end{eqnarray*} 
This concludes our proof.
\end{proof}

The formula obtained for the higher order derivatives of $K_{\chi}(A)(X^1,\ldots ,X^k)$ generalises the expressions obtained by Bhatia, Jain and Grover (\cite{BJ}, \cite{G}). We will prove this for the derivative of the $m$-th compound, establishing that, from the formula in theorem \ref{der_pot}, one can establish formula (2.5) in \cite{J}, from which the main formula for the derivative of the $m$-th compound of $A$ is obtained. 

So we take $\chi= \varepsilon$ then $$K_{\chi}(A)(X^1,\ldots ,X^k)=\wedge^m (A)(X^1,\ldots ,X^k).$$
In this case $\widehat{\Delta}=Q_{m,n}$ and the basis $\lbrace e^\wedge_ \alpha : \alpha \in Q_{m,n} \rbrace$ is orthogonal and it is easy to see (by direct computation or using formula \eqref{norma-tensores}) that every vector has norm $1/\sqrt{m!}$. So the matrix $B$ of order $\binom{n}{m}$ is diagonal and its diagonal entries are equal to $\sqrt{m!}$. 

We now notice two properties that we will use in our computations:\medskip

I. For any square matrices $X\in M_k(\C)$, $Y\in M_{n-k}(\C)$ and functions $\alpha,\beta\in Q_{k,n}$,
$$\det X\bigoplus_{\alpha|\beta} Y = (-1)^{|\alpha|+|\beta|} \det X \det Y.$$
This is a consequence of formula \eqref{gen-det-laplace}. We again notice that if $\gamma\neq \beta$, the matrices 
$$(X\bigoplus_{\alpha|\beta} Y)[\alpha|\gamma] \text{ and } (X\bigoplus_{\alpha|\beta} Y)(\alpha|\gamma)$$
have a zero column. Now, using the Laplace expansion for the determinant along $\alpha$,
\begin{eqnarray*}
\det X\bigoplus_{\alpha|\beta} Y & = & (-1)^{|\alpha|}\sum_{\gamma\in Q_{k,n}} (-1)^{|\gamma|} \det((X\bigoplus_{\alpha|\beta} Y)[\alpha|\gamma])\det((X\bigoplus_{\alpha|\beta} Y)(\alpha|\gamma))\\
& = & (-1)^{|\alpha|+|\beta|} \det((X\bigoplus_{\alpha|\beta} Y)[\alpha|\beta])\det((X\bigoplus_{\alpha|\beta} Y)(\alpha|\beta))\\
& = & (-1)^{|\alpha|+|\beta|} \det X\det Y.
\end{eqnarray*}

II. For $\alpha,\beta\in Q_{m,n}$ and $\rho,\tau\in Q_{k,m}$, we have 
$$\sum_{\sigma \in S_k} \det(X[ \alpha | \beta]^{\sigma}_{\tau} [ \rho| \tau]) = 
k! \Delta(X^1[ \alpha | \beta]  [ \rho| \tau], \ldots,X^k[ \alpha | \beta]  [ \rho| \tau])$$
To check this, consider the columns of the matrices involved. Remember that 
$$X[ \alpha | \beta]^{\sigma}_{\tau} = \mathbf{0}(\tau, X^{\sigma(1)}[\alpha|\beta],\ldots,X^{\sigma(k)}[\alpha|\beta]).$$

For given $\sigma\in S_k$ and $j\in [k]$, we have:\medskip

\begin{tabular}{rcl}
entry $(i,j)$ of $X[ \alpha | \beta]^{\sigma}_{\tau}[\rho|\tau]$ & = & entry $(\rho(i),\tau(j))$ of $X[ \alpha | \beta]^{\sigma}_{\tau}$\medskip\\
& = & entry $(\rho(i),\tau(j))$ of $X^{\sigma(j)}[\alpha | \beta]_{\tau(j)}$\medskip \\
& = & entry $(i,j)$ of $X^{\sigma(j)}[ \alpha | \beta][\rho|\tau]$.\medskip
\end{tabular}

Therefore, 
$$X[ \alpha | \beta]^{\sigma}_{\tau}[\rho|\tau] = [X^{\sigma(1)}[\alpha|\beta][\rho|\tau]_{[1]}\, \ldots \, X^{\sigma(k)}[\alpha|\beta][\rho|\tau]_{[k]})$$
and the matrices that appear in the first sum are the same as the ones that appear in the mixed discriminant. 
\medskip

We are now ready to prove the result. If we replace in theorem \ref{der_pot} $d_\chi=\det$, we have that that the $(\alpha,\beta)$ entry of $D^k \wedge^m (A)(X^1,\ldots ,X^k)$ is
\begin{eqnarray*}
&& \frac{1}{m!}\displaystyle \sum_{\gamma ,  \delta \in Q_{m,n} }b_{\gamma\beta} \overline{b_{\delta\alpha}}\displaystyle \sum_{\sigma \in S_k}\displaystyle \sum_{\rho, \tau \in Q_{k,m}}\det(X[\delta|\gamma]^{\sigma}_{\tau} [ \rho| \tau]\bigoplus_{\rho| \tau} A[\delta|\gamma](\rho | \tau))\\
&=&\frac{1}{m!}m! \sum_{\sigma \in S_k} \sum_{\rho, \tau \in Q_{k,m}}\det(X[\alpha | \beta]^{\sigma}_{\tau}  [ \rho| \tau]\bigoplus_{\rho| \tau} A[ \alpha | \beta](\rho | \tau))\\
&=&  \sum_{\sigma \in S_k}\displaystyle \sum_{\rho, \tau \in Q_{k,m}} (-1)^{|\rho|+|\tau|} \det(A[ \alpha | \beta](\rho | \tau)) \det(X[ \alpha | \beta]^{\sigma}_{\tau} [ \rho| \tau])\\
&=& k!\sum_{\rho, \tau \in Q_{k,m}} (-1)^{|\rho|+|\tau|} \det(A[ \alpha | \beta](\rho | \tau)) \Delta(X^1[ \alpha | \beta]  [ \rho| \tau], \ldots,X^k[ \alpha | \beta]  [ \rho| \tau])
\end{eqnarray*}

We denoted by $\Delta(B_1,...,B_n)$ the mixed discriminant. The formula we obtained is formula (2.5) in \cite{J}, if you take into account that in this paper the roles of the letters $k$ and $m$ are interchanged.\medskip

Using similar arguments one can obtain the formula for the $k$-th derivative of $\vee^m (A)(X^1,\ldots ,X^k)$ in \cite{G}.\medskip

{\bf Acknowledgements}. We would like to thank Prof.\ Rajendra Bhatia for presenting this problem to us and for all the suggestions and discussions we had on this subject. We would also like to thank Prof.\ Jos\'e Dias da Silva for sharing with us his understanding of these topics and for all the discussions we had. 
   
\bibliographystyle{plain}

\end{document}